%% file: main.tex
\documentclass[letterpaper]{article} 
\usepackage{authblk}
\usepackage{arxiv}  
\usepackage{times}  
\usepackage{helvet}  
\usepackage{courier}  
\usepackage[hyphens]{url}  
\usepackage{graphicx} 
\usepackage{natbib}
\usepackage[table]{xcolor}
\usepackage{color-edits}
\addauthor[Hui]{hui}{blue}
\usepackage{float} 
\usepackage{xspace}
\usepackage{mathtools}
\usepackage{bm}
\usepackage{wrapfig}
\usepackage{booktabs}
\usepackage{tikz}
\usepackage[utf8]{inputenc} 
\usepackage[T1]{fontenc}    
\usepackage{booktabs}       
\usepackage{amsfonts}       
\usepackage{nicefrac}       
\usepackage{microtype}      
\usepackage{lipsum}         
\usepackage{subcaption}
\usepackage{amsmath}
\usepackage{algorithm}
\usepackage{algpseudocode}
\usepackage{pgfgantt}
\usepackage{appendix}
\usepackage{mathrsfs}
\usepackage{multirow}

\usepackage{adjustbox}
\usepackage{threeparttable}
\usepackage{hyperref}       
\usepackage{tabularx}
\usepackage{amsthm, comment}
\usepackage{pifont}
\usepackage{enumitem}
\makeatletter
\newsavebox{\@tabnotebox}

\makeatother

\newcommand{\name}{PDHCG-II\xspace}
\title{\name: An Enhanced Version of PDHCG for Large-Scale Convex QP}
\newtheorem{theorem}{Theorem}[section] 

\theoremstyle{definition}

\usepackage{siunitx}
\theoremstyle{remark}
\author[1]{Hongpei Li}
\author[1]{Yicheng Huang}
\author[2]{Huikang Liu}
\author[2]{Dongdong Ge}
\author[3]{Yinyu Ye}

\affil[1]{Shanghai University of Finance and Economics}
\affil[2]{Shanghai Jiao Tong University}
\affil[3]{Stanford University}

\AtBeginDocument{%
  \addtolength\abovedisplayskip{-0.05\baselineskip}%
  \addtolength\belowdisplayskip{-0.1\baselineskip}%
  \addtolength\abovedisplayshortskip{-0.05\baselineskip}%
  \addtolength\belowdisplayshortskip{-0.1\baselineskip}%
}

\usepackage{setspace}
\usepackage{titlesec}
\titlespacing*{\section}{0pt}{-0.05\baselineskip}{-0.05\baselineskip}
\titlespacing*{\subsection}{0pt}{-0.075\baselineskip}{-0.075\baselineskip}
\titlespacing*{\subsubsection}{0pt}{-0.025\baselineskip}{-0.05\baselineskip}
\begin{document}
\maketitle

\input{text/abs}

\input{text/intro}
\input{text/algo}
\input{text/enhance}
\input{text/kkt}
\input{text/infeasibility}
\input{text/exp}
\bibliographystyle{plainnat}
\bibliography{main,ref}

\appendix
\end{document}

%% file: text/abs.tex
\begin{abstract}
Quadratic programming (QP) is a fundamental optimization model with wide-ranging applications in decision-making and machine learning, yet efficiently solving large-scale instances remains a major computational challenge.
Building upon the recently developed PDHCG framework, we propose \textbf{\name}, an enhanced first-order solver tailored for large-scale convex QPs. The proposed method explicitly exploits the quadratic structure of the objective and incorporates several key algorithmic innovations, including Halpern-type acceleration and a PID-controlled adaptive update of the primal–dual weight. To further improve practical performance, \name introduces a refined adaptive termination criterion for inner subproblems to prevent over-solving, together with an infeasibility detection mechanism for robust handling of ill-posed instances. Extensive numerical experiments demonstrate that \name consistently achieves $2.5\times$–$5\times$ speedups over PDHCG on standard QP benchmarks. To facilitate reproducibility and broader adoption, we release a CUDA-C implementation of \name as open-source software\footnote{\url{https://github.com/Lhongpei/PDHCG-II}}.
\end{abstract}

%% file: text/intro.tex
\section{Introduction}

Convex quadratic programming (QP) is a core problem in mathematical optimization and underpins a wide range of applications in machine learning, control, finance, and operations research, including support vector machines (SVMs) \citep{hearst1998support}, model predictive control (MPC) \citep{garcia1989model,morari1999model}, and portfolio optimization \citep{markowitz1952portfolio}. As modern applications increasingly involve large datasets and high-dimensional models, there is a growing demand for solvers capable of handling large-scale QP instances, often featuring millions of variables and constraints.

Interior-point methods (IPMs) \citep{vanderbei1999loqo} have long been regarded as the workhorse for QP due to their strong convergence guarantees and high solution accuracy. However, each IPM iteration requires solving large linear systems, typically involving costly matrix factorizations and substantial memory overhead, which severely limits their scalability and makes them unsuitable for very large problems or hardware-constrained environments. These challenges have renewed interest in first-order methods (FOMs), such as the alternating direction method of multipliers (ADMM) \citep{boyd2011distributed} and the primal-dual hybrid gradient (PDHG) method \citep{chambolle2011first}. By relying primarily on matrix–vector products, FOMs can better exploit problem sparsity and modern parallel architectures, including GPUs, making them an attractive alternative for large-scale QP.

Despite their scalability, standard first-order methods (FOMs) often suffer from slow convergence on ill-conditioned problems and exhibit pronounced stagnation when high-accuracy solutions are required. While substantial progress has been achieved for large-scale linear programming (LP), most notably through the PDLP solver \citep{applegate2021practical}, extending comparable performance gains to general QP remains challenging due to the more intricate geometry induced by the quadratic objective.

Recent efforts toward scalable QP solvers include PDQP \citep{lu2023practical}, PDHCG \citep{huang2025restarted}, and HPR-QP \citep{chen2025hpr}. Among these methods, PDHCG demonstrates particularly strong performance on large-scale instances by incorporating a conjugate gradient (CG) solver for the primal subproblem, thereby effectively addressing the quadratic term. However, while PDHCG effectively mitigates issues related to the condition number of the quadratic matrix, it exhibits instability in other aspects, which limits its overall robustness and effectiveness on general instances.

In this paper, we propose \textbf{\name}, an enhanced variant of PDHCG that bridges scalability and high-precision robustness through systematic algorithmic improvements. Our approach integrates modern acceleration techniques with adaptive control mechanisms. The main contributions are summarized as follows:
\begin{itemize}
\item \textbf{Algorithmic Enhancements.} We incorporate a Halpern-type acceleration scheme and a proportional–integral–derivative (PID) controller to adaptively balance primal and dual step sizes, proposed in~\citep{lu2025cupdlpx}. Besides, we propose a novel adaptive stopping criterion for primal subproblems for preventing over-solving. These components collectively achieve a substantial reduction in iteration counts across a wide range of problem instances.

\item \textbf{Solver Completeness and Efficiency.} Our proposed solver features a complete termination framework, incorporating criteria for detecting primal and dual infeasibility. Additionally, we introduce a specialized operator to handle quadratic objectives matrix, which is composed of sparse and low-rank matrices, allowing for efficient computation without the need to explicitly materialize dense matrices.


\item \textbf{High-Performance Open-Source Implementation.} We provide highly optimized implementations in both Julia and CUDA-C, designed to fully exploit GPU parallelism while reducing overhead. Extensive experiments on standard benchmarks (Maros–Mészáros and Mittelmann) show that \name consistently outperforms strong baselines such as HPR-QP and PDQP, particularly in high-accuracy regimes ($\epsilon = 10^{-8}$). The CUDA-C implementation is released as open-source software to facilitate reproducibility and community adoption.
\end{itemize}

%% file: text/algo.tex
\section{Algorithm Framework}
\label{algo}

We consider the following convex quadratic programming (QP) problem:
\begin{equation}
\begin{aligned}
    \min_{x \in \mathcal{X}} \quad & \frac{1}{2} x^\top Q x + c^\top x \\
    \text{s.t.} \quad & A x \in \mathcal{S}, \label{eq:QP}
\end{aligned}
\end{equation}
where $\mathcal{X} := \{ x \in \mathbb{R}^n \mid l_v \le x \le u_v \}$ denotes box constraints with componentwise bounds $l_v \in (\mathbb{R}\cup\{-\infty\})^n$ and $u_v \in (\mathbb{R}\cup\{+\infty\})^n$, $Q \succeq 0$ is a symmetric positive semidefinite matrix, $A \in \mathbb{R}^{m\times n}$ is the constraint matrix, and $\mathcal{S} := \{ s \in \mathbb{R}^m \mid l_c \le s \le u_c \}$ represents affine constraints with bounds $l_c \in (\mathbb{R}\cup\{-\infty\})^m$ and $u_c \in (\mathbb{R}\cup\{+\infty\})^m$.

Unlike LP, convex QP does not in general admit a dual formulation with the same strong structural properties. A Wolfe-type dual formulation for nonlinear programming was introduced in \citep{wolfe1961duality}. In this work, we adopt the restricted Wolfe dual form proposed in \citep{li2018qsdpnal}, given by
\begin{equation}
\begin{aligned}
    \max_{x,\,y,\,r} \quad & -p(-r; l_v, u_v) - \frac{1}{2} x^\top Q x - p(y; l_c, u_c) \\
    \text{s.t.} \quad & r = Qx + A^\top y + c,\quad y \in \mathcal{Y},\quad r \in \mathcal{R}, \label{eq:dualQP}
\end{aligned}
\end{equation}
where
\[
p(z; l, u) := u^\top z^+ - l^\top z^-, \qquad
z^+ = \max\{z,0\}, \quad z^- = \max\{-z,0\}.
\]
The sets $\mathcal{Y}$ and $\mathcal{R}$ encode sign restrictions induced by the bound constraints:
\[
\mathcal{Y}_i =
\begin{cases}
\{0\} & (l_c)_i = -\infty,\ (u_c)_i =\infty,\\
\mathbb{R}^+, & (l_c)_i = -\infty,\ (u_c)_i \in \mathbb{R},\\
\mathbb{R}^-, & (l_c)_i \in \mathbb{R},\ (u_c)_i = +\infty,\\
\mathbb{R},   & (l_c)_i \in \mathbb{R},\ (u_c)_i \in \mathbb{R},
\end{cases}
\qquad
\mathcal{R}_i =
\begin{cases}
\{0\} & (l_v)_i = -\infty,\ (u_v)_i =\infty,\\
\mathbb{R}^-, & (l_v)_i = -\infty,\ (u_v)_i \in \mathbb{R},\\
\mathbb{R}^+, & (l_v)_i \in \mathbb{R},\ (u_v)_i = +\infty,\\
\mathbb{R},   & (l_v)_i \in \mathbb{R},\ (u_v)_i \in \mathbb{R}.
\end{cases}
\]
Equivalently, the primal problem \eqref{eq:QP} can be expressed in saddle-point form as
\begin{equation}\label{eq:minimax}
    \min_{x \in \mathcal{X}} \max_{y \in \mathcal{Y}}
    \quad
    \frac{1}{2} x^\top Q x + c^\top x + y^\top A x - p(y; l_c, u_c),
\end{equation}
which provides the foundation for the primal--dual algorithm developed in the following sections.

\subsection{The PDHG Iteration Scheme}
PDHG solves the associated saddle-point problem \eqref{eq:minimax} via alternating primal descent and dual ascent steps. Specifically, given the current iterate $(x^k, y^k)$, the PDHG updates are defined by
\begin{subequations}\label{eq:pdhcg_update}
\begin{align}
    x^{k+1} &= \arg\min_{x \in \mathcal{X}} \left\{ \frac{1}{2}x^\top Q x + c^\top x + x^\top A^\top y^k + \frac{1}{2\tau_k} \|x - x^k\|_2^2 \right\}, \label{eq:primal_update} \\
    \bar{x}^{k+1} &= 2x^{k+1} - x^k, \label{eq:extrapolation} \\
    y^{k+1} &= y^k + \sigma_k A \bar{x}^{k+1} - \sigma_k \cdot \text{proj}_{\mathcal{S}}\!\left( \frac{1}{\sigma_k} y^k + A \bar{x}^{k+1} \right), \label{eq:dual_update}
\end{align}
\end{subequations}
where $\tau_k$ and $\sigma_k$ denote the primal and dual stepsizes, respectively. The extrapolated variable $\bar{x}^{k+1}$ corresponds to the standard over-relaxation step in PDHG-type methods. The dual update \eqref{eq:dual_update} can be interpreted as a projected gradient ascent step with respect to the extrapolated primal iterate.

\subsection{Primal Subproblem Solvers} \label{sec:inner_solvers}

The overall efficiency of \name critically depends on the rapid solution of the primal subproblem \eqref{eq:primal_update}, which is a strictly convex quadratic program subject to box constraints over $\mathcal{X}$. Depending on the structure of the quadratic matrix $Q$ and the presence of bound constraints, we adopt different solution strategies:

\begin{itemize}
    \item \textbf{Diagonal Hessian.} When $Q$ is diagonal, the primal subproblem admits a closed-form solution obtained via componentwise operations.

    \item \textbf{General Case.} For the fully constrained setting, we apply a projected gradient method equipped with Barzilai--Borwein (BB) step sizes.
\end{itemize}

\paragraph{Structured Quadratic Objective.}
A key feature of \name is its ability to efficiently handle structured quadratic objectives. Specifically, we allow the matrix $Q$ to admit a decomposition of the form
\begin{equation}
    Q = P + R^\top R,
\end{equation}
where $P$ is sparse and $R \in \mathbb{R}^{k \times n}$ (with $k \ll n$) represents a low-rank factor. Such decompositions commonly arise in applications including data science and finance, for example in portfolio optimization with factor models.

Rather than forming the dense matrix $R^\top R$ explicitly, our implementation performs matrix–vector products with $P$ and $R$ separately, yielding a matrix-free realization of the quadratic term. This design substantially reduces memory usage and avoids unnecessary dense computations, thereby enabling the efficient solution of large-scale QP instances that would otherwise be infeasible on memory-limited hardware.

%% file: text/enhance.tex
\section{Algorithmic Enhancements}

In this section, we describe several algorithmic enhancements incorporated into \name to improve efficiency and convergence behavior beyond the baseline PDHCG framework.

\subsection{Reflected Halpern Scheme}

The Halpern scheme is an acceleration technique for PDHG-type methods that improves convergence behavior by introducing an explicit anchoring to the initial iterate. It has recently been shown to be effective for large-scale linear programming, offering both strong theoretical guarantees and encouraging empirical performance \citep{lu2025cupdlpx}.

Halpern PDHG constructs each new iterate as a weighted combination of the current PDHG update and the initial point. In particular, letting $z=(x,y)$ denote the primal--dual variable, the reflected Halpern update at iteration $k$ is given by
\begin{equation}
    z^{k+1}
    =
    (1+\theta)
    \left(
    \frac{k+1}{k+2}\,\text{PDHG}(z^k)
    +
    \frac{1}{k+2}\,z^0
    \right)
    -
    \theta z^{k-1},
\end{equation}
where $\text{PDHG}(z^k)$ denotes a single PDHG step applied to the current iterate $(x^k,y^k)$ as defined in \eqref{eq:pdhcg_update}, and $\theta$ is a reflection parameter.

As discussed in \citep{lu2024restarted}, the Halpern scheme shares structural similarities with averaged and restarted variants of PDHG, but differs in its anchoring mechanism and theoretical convergence properties. In \name, we adopt this reflected Halpern formulation to accelerate convergence while preserving the stability of the underlying primal--dual updates.

\subsection{Adaptive Primal--Dual Weight}

Following effective practices in modern first-order solvers \citep{lu2025cupdlpx}, we reparameterize the stepsizes $\tau_k$ and $\sigma_k$ in terms of a global stepsize $\eta_k$ and a primal--dual weight $\omega_k$:
\begin{equation}
    \tau_k = \frac{\eta_k}{\omega_k}, \qquad \sigma_k = \eta_k \omega_k.
\end{equation}
Here, $\eta_k$ governs the overall convergence speed, while $\omega_k$ controls the relative scaling between the primal and dual updates. This decomposition enables independent adaptivity of the two parameters. In our implementation, we fix $\eta_k = 0.998/\|A\|$ and dynamically adjust $\omega_k$ using a proportional--integral--derivative (PID) controller to maintain primal--dual balance, which has been applied in first-order linear programming solvers~\citep{lu2025cupdlpx}.

Specifically, we define the imbalance measure of $n$-th restarted round as
\begin{equation}
    e_n = \log \left( 
    \frac{\sqrt{\omega_n}\,\| x^{n,t} - x^{n,0} \|_2}
    {\frac{1}{\sqrt{\omega_n}}\,\| y^{n,t} - y^{n,0} \|_2}
    \right),
\end{equation}
where $(x^{n,0}, y^{n,0})$ and $(x^{n,t}, y^{n,t})$ denote the first and last iterate of $n$-th restarted round, and update the primal--dual weight according to
\begin{equation}\label{eq:pid}
    \log \omega_{n+1} = \log \omega_n -
    \Bigl[ K_P e_n + K_I \sum_{i=1}^n e_i + K_D (e_n - e_{n-1}) \Bigr],
\end{equation}
where $K_P$, $K_I$, and $K_D$ denote the proportional, integral, and derivative gains, respectively.

\subsection{Adaptive Stopping Criterion for Primal Subproblems}

In \name, the primal update associated with the quadratic objective is solved inexactly using either the conjugate gradient (CG) or Barzilai--Borwein (BB) method. Employing a fixed termination tolerance for this inner solver is often inefficient: a loose tolerance may hinder convergence, while an overly stringent tolerance leads to unnecessary computational cost, especially during the early iterations. To address this issue, we introduce a dynamic, heuristic stopping criterion based on the progress of the outer iterations.

Rather than relying on strict theoretical energy bounds, we dynamically scale the inner tolerance based on the primal movement. To ensure numerical stability and consistent convergence, we explicitly enforce a monotonically decreasing tolerance across iterations. Specifically, the allowable residual norm threshold $\epsilon_k$ for the inner solver at iteration $k$ is updated according to the following practical rule:
\begin{equation}
    \epsilon_k = \min \left( \epsilon_{k-1}, \max \left( \gamma \frac{\omega \|x_k - x_{k-1}\|}{\tau}, \epsilon_{\min} \right) \right),
\end{equation}
where $\|x_k - x_{k-1}\|$ represents the primal movement, $\tau$ is the step size, $\omega$ is the primal weight, and $\gamma$ is a constant scaling factor. In our implementation, we set $\gamma = 5 \times 10^{-4}$ and impose a lower bound $\epsilon_{\min} = 10^{-9}$ to prevent numerical underflow. 

This heuristic ensures that the inner subproblem is solved with increasing precision only as the outer sequence converges (i.e., as the primal movement diminishes), while the monotonic $\min$ operator prevents the tolerance from inadvertently relaxing if the step difference temporarily spikes. Because all required components are readily available within the main algorithmic loop, this adaptive criterion is highly lightweight and introduces no additional computational overhead.

%% file: text/kkt.tex
\section{Optimality Conditions and Stopping Criteria}

In this section, we derive the optimality conditions for the convex QP problem \eqref{eq:QP} and define the specific residual metrics used to monitor convergence and terminate the solver.

\subsection{KKT Conditions}
The optimality conditions for the saddle-point problem \eqref{eq:minimax} can be characterized using variational inequalities. A pair $(x^\star, y^\star) \in \mathcal{X} \times \mathcal{Y}$ is optimal if and only if it satisfies:
\begin{subequations}
\begin{align}
    \text{Primal Feasibility:} \quad & Ax^\star \in \mathcal{S}, \\
    \text{Dual Feasibility:} \quad & r^\star =Q x^\star + c + A^\top y^\star \in \mathcal{R}, \\
    \text{Duality Gap:} \quad & (x^\star)^\top Qx^\star + c^\top x^\star +p(-r^\star; l_v, u_v) + p(y^\star; l_c, u_c) = 0.
\end{align}
\end{subequations}

To numerically measure the violation of these conditions, we adopt the \textit{natural residual} formulation common in proximal gradient methods. Specifically, the violation of the stationarity condition is measured by the distance between the current iterate $x$ and the result of a proximal gradient step.

\subsection{Termination Criteria}
\label{subsection: termination}
We define the normalized primal residual ($r_{\text{primal}}$), dual residual ($r_{\text{dual}}$), and duality gap ($r_{\text{gap}}$) based on the $L_\infty$-norm, which provides a stricter convergence criterion than the $L_2$-norm typically used in theoretical analysis. The solver terminates when the maximum of these residuals falls below a user-specified tolerance $\epsilon_{\text{tol}}$ (e.g., $10^{-4}$ or $10^{-6}$):
\begin{equation}
    \max \{ r_{\text{primal}}, r_{\text{dual}}, r_{\text{gap}} \} \le \epsilon_{\text{tol}}.
\end{equation}
The specific formulations for the residuals are given below.

\paragraph{Primal Residual}
The primal residual measures the violation of the linear constraints $Ax \in \mathcal{S}$ relative to the scale of the constraint bounds:
\begin{equation}
    r_{\text{primal}} = \frac{\left\|Ax-\operatorname{proj}_{\mathcal{S}}(Ax)\right\|_\infty}{1+\max \left\{\left\|l_{c}\right\|_\infty,\left\|u_{c}\right\|_\infty\right\}}.
\end{equation}

\paragraph{Dual Residual}
The dual residual quantifies the violation of the dual feasibility $r =Q x + c + A^\top y \in \mathcal{R}$. The normalized dual residual is then defined as the magnitude of the violation, scaled by the problem data:
\begin{equation}
    r_{\text{dual}} = \frac{\left\|r -  \text{proj}_{ \mathcal{R}} \left(r \right) \ \right\|_\infty}{1 + \max\{\|Qx\|_\infty, \|A^\top y\|_\infty, \|c\|_\infty \}}.
\end{equation}

\paragraph{Duality Gap}
The duality gap provides a comprehensive measure of optimality, combining primal and dual objective values. We define the normalized gap as:
\begin{equation}
    r_{\text{gap}}=\frac{\left|x^\top Qx + c^\top x +p(-r; l_v, u_v) + p(y; l_c, u_c)\right|}{1+\max\left\{\left|\frac{1}{2}x^\top Qx+c^\top x\right|,\left| \frac{1}{2} x^\top Qx  +p(-r; l_v, u_v) + p(y; l_c, u_c) \right|\right\}}.
\end{equation}

%% file: text/infeasibility.tex
\section{Infeasibility Detection}

In addition to solving for optimal primal-dual pairs, \name can also detect when the problem is infeasible or unbounded. This corresponds to identifying certificates of \textit{primal infeasibility} or \textit{dual infeasibility} (unboundedness). By the theorem of alternatives, the non-existence of a primal (dual) feasible solution is certified by the existence of a \textit{ray} such that the auxiliary dual (primal) objective tends to infinity while satisfying constraints. Inspired by recently study on the behavior of FOM in infeasible cases, such as ALM \citep{andrews2025augmented} and PDLP \citep{applegate2024infeasibility}, one could extract the Farkas ray from the update iterations. In our implementation, candidate certificates rays are extracted from the difference between consecutive iterates of the solver, denoted as $\Delta y$ and $\Delta x$.

\subsection{Primal Infeasibility}

Primal infeasibility means the following LP problem is infeasible:
$$
\min_{x \in \mathcal{X}} \; 0 \quad \text{s.t.} \; Ax \in \mathcal{S},
$$
which corresponds the case $Q =0$ and $c = 0$ in \eqref{eq:QP}. According to \eqref{eq:dualQP}, its dual problem is given by
$$
\max_{y \in \mathcal{Y}} \; -p(-A^\top y; l_v, u_v) - p(y; l_c, u_c) \quad \text{s.t.} \; A^\top y \in \mathcal{R}.
$$
So the infeasibility of primal problem means there exists a solution to the following system
$$
y \in \mathcal{Y}, \quad A^\top y \in \mathcal{R}, \quad -p(-A^\top y; l_v, u_v) - p(y; l_c, u_c) > 0.
$$

In \name, we construct a candidate dual ray $\tilde{y}$ by projecting the iterate difference $\Delta y$ onto the recession cone of the dual variables, that is
$$
\tilde{y} = \text{proj}_{\mathcal{Y}} (\Delta y).
$$
The solver terminates with status \texttt{PRIMAL\_INFEASIBLE} if the ray yields a positive improving direction for the objective $-p(-A^\top y; l_v, u_v) - p(y; l_c, u_c)$ and the violation of the dual constraints is negligible. In practice, the termination criteria is set as
\begin{equation}
   \|A^\top  \tilde{y} - \text{proj}_{\mathcal{R}} (A^\top  \tilde{y}) \|_\infty \le \epsilon_{\text{inf}} \cdot \left[p(-A^\top \tilde{y}; l_v, u_v) + p(\tilde{y}; l_c, u_c)\right]^-
\end{equation}
for some small enough $\epsilon_{\text{inf}} > 0$.

\subsection{Dual Infeasibility}

A problem is dual infeasible (or primal unbounded) if there exists a direction $d$ along which the objective function tends to $-\infty$ while remaining feasible. For a generic Convex QP, a certificate of unboundedness must satisfy:
\begin{subequations}\label{eq:dualinfea}
\begin{align}
d \in \text{Recc}(\mathcal{X}) \quad & \text{(Region Feasibility)}\\
    A d \in \text{Recc}(\mathcal{S}) \quad & \text{(Constraint Feasibility)} \\
    c^\top  d < 0 \quad & \text{(Descent Direction)} \\
    Q d = 0 \quad & \text{(Vanishing Curvature)}
\end{align}
\end{subequations}
where $\text{Recc}(\mathcal{Z})$ represents the recession cone of any convex set $\mathcal{Z}$, which is defined as
$$
\text{Recc}(\mathcal{Z}) \coloneqq \left\{ z \in \mathbb{R}^m \mid \forall s \in \mathcal{S}, \, \forall \lambda >0, s+ \lambda z \in \mathcal{Z}\right\}.
$$

\begin{theorem}
The convex quadratic program \eqref{eq:QP} is unbounded if and only if \eqref{eq:dualinfea} holds for some $d \in \mathbb{R}^n$.
\end{theorem}
\begin{proof}
($\Leftarrow$)
Let $x_0$ be an arbitrary feasible point and suppose $d$ satisfies \eqref{eq:dualinfea}.
Since $d \in \text{Recc}(\mathcal{X})$ and $Ad \in \text{Recc}(\mathcal{S})$, we know that $x_0+\lambda d \in \mathcal{X}$ and $A(x_0+\lambda d)\in \mathcal{S}$, so $x_0+\lambda d$ is also feasible for all $\lambda\ge0$.
Moreover, $Qd=0$ and $c^\top d < 0$ imply that the objective value 
\[
f(x_0+ \lambda d)
=
f(x_0)+\lambda \,c^\top d \to -\infty
\quad (\lambda \to\infty)
\]
which proves unboundedness.

($\Rightarrow$)
Assume \eqref{eq:QP} is unbounded.
Then there exists a feasible sequence $\{x_k\}$ with $f(x_k)\to-\infty$.
For any feasible point $\bar x$, let
$x_k=\bar x+t_k d_k$, where $t_k=\|x_k-\bar x\|\to\infty$ and $\|d_k\|=1$.
By compactness, take a subsequence with $d_k\to d\neq0$.

From feasibility, we claim that $d \in \text{Recc}(\mathcal{X})$ because of the following. For any $\lambda \geq 0$, there exists some $k_0$ such that $t_k \geq \lambda$ for all $k \geq k_0$, then the fact that $\bar{x} \in \mathcal{X}$ and $\bar{x} + t_k d_k \in \mathcal{X}$ implies that 
\[
\bar{x} + \lambda d_k \in \mathcal{X}
\]
holds for all $k \geq k_0$. By taking $k \rightarrow \infty$, we know that $\bar{x} + \lambda d\in \mathcal{X}$ for all $\lambda \geq 0$ and all $\bar{x} \in \mathcal{X}$, which completes the proof of the claim. Similarly, we could show that $Ad \in \text{Recc}(\mathcal{S})$. 

Expanding the objective,
\[
f(\bar x+t_k d_k)
=
\tfrac12 t_k^2 d_k^\top Q d_k
+ t_k(d_k^\top Q\bar x+c^\top d_k)
+ O(1).
\]
Dividing by $t_k^2$ and letting $k\to\infty$ yields
$d^\top Q d=0$, hence the fact that $Q \succeq 0$ implies that $Qd=0$.
Since the quadratic term vanishes and $f(x_k)\to-\infty$, the linear term must be strictly negative, giving $c^\top d<0$.
\end{proof}

In our implementation, the candidate primal ray $\tilde{d}$ is obtained by normalizing the projected iterate difference $\Delta x$, i.e., $ \tilde{d} = \Delta x / \| \Delta x \|_\infty$. The solver monitors a combined residual that accounts for both linear constraint violation and non-zero quadratic curvature. The problem is declared \texttt{DUAL\_INFEASIBLE} if
\begin{equation}
    c^\top  \tilde{d} < -\epsilon_{\text{tol}} \quad \text{and} \quad  \max \left\{ \|\tilde{d} - \text{proj}_{\text{Recc}(\mathcal{X})}( \tilde{d}) \|_\infty, \; \| A \tilde{d} - \text{proj}_{\text{Recc}(\mathcal{S})}(A \tilde{d}) \|_\infty, \; \frac{\| Q \tilde{d} \|_\infty}{\gamma} \right\} \le \epsilon_{\text{inf}}.
\end{equation}
Here, $\gamma$ is a system scaling factor.

%% file: text/exp.tex
\section{Numerical Experiment}
In this section, we present numerical experiments to evaluate the performance of the proposed PDHCG-II algorithm. We start by describing the implementation details and experimental setup, including the computing environment and termination criteria. We then report the computational results on various datasets, comparing PDHCG-II against state-of-the-art QP solvers to demonstrate its efficiency and robustness.
\subsection{Implementation Details}\label{sec:implementation}

\paragraph{Solver Implementation} We implement \name in both Julia and C to balance research flexibility with production efficiency. \textbf{\name.jl} is a Julia package leveraging \texttt{CUDA.jl} and just-in-time compilation, designed for rapid prototyping while maintaining competitive performance. Complementing this, \textbf{\name-C} is a highly optimized C implementation using the direct CUDA C API, intended for production environments where minimizing overhead and integration into larger systems are critical.

\paragraph{Computing Environment} Experiments are conducted on a server with an NVIDIA H100 GPU (80GB HBM3), an Intel Xeon Platinum 8469C CPU (2.60 GHz), and 512 GB RAM.

\paragraph{Termination Condition} In all experiments, we enforce the strict termination criteria defined in Section~\ref{subsection: termination}. The solver terminates successfully only when the relative primal residual ($r_{\text{primal}}$), dual residual ($r_{\text{dual}}$), and duality gap ($r_{\text{gap}}$) all drop below the specified tolerance threshold $\epsilon$. Formally, the convergence condition is met when:
\begin{equation}
    \max(r_{\text{primal}}, r_{\text{dual}}, r_{\text{gap}}) \le \epsilon,
\end{equation}
where the detailed definitions of these $L_\infty$-norm based residuals are provided in Section~\ref{subsection: termination}.

\subsection{Benchmark QP Datasets}
We evaluate the performance of PDHCG-II on standard benchmark repositories:
\begin{enumerate}
    \item \textbf{Maros–Mészáros Benchmark} \cite{maros1999repository}: A classic collection containing \textbf{134 instances} of convex quadratic programming problems. For this dataset, we conduct all experiments with a standard termination tolerance of $\epsilon = 10^{-6}$.
    \item \textbf{Mittelmann's Benchmark} \cite{mittelmann2021decision}:There are 21 instances from the test set maintained by Hans Mittelmann, tailored for benchmarking continuous convex qp solvers. On this benchmark, we evaluate solvers at both $\epsilon = 10^{-6}$ and $\epsilon = 10^{-8}$ precision levels. It is important to note that for the high-precision experiments ($\epsilon = 10^{-8}$), we specifically compare PDHCG-II against HPR-QP, which demonstrated the second-best performance under $\epsilon = 10^{-6}$, while PDQP and the original PDHCG are evaluated only at $\epsilon = 10^{-6}$.
\end{enumerate}
\paragraph{Evaluation Metrics} 
We evaluate solver performance based on the number of solved instances, success rate, arithmetic mean time, and Shifted Geometric Mean (SGM) with a shift of 10. To strictly penalize failures, if an instance does not converge within the specified time limit or reaches the maximum iteration count, its runtime $t_i$ is recorded as the time limit itself (e.g., $1000s$ or $3600s$). The $\text{SGM}{10}$ is then defined as:
\begin{equation}
    \text{SGM}_{10} = \exp\left(\frac{1}{n} \sum_{i=1}^{n} \ln(t_i + 10)\right) - 10,
\end{equation}
where $t_i$ represents the runtime for instance $i$. Unlike the arithmetic mean, the SGM reduces the influence of extreme outliers, providing a more balanced comparison of solver efficiency across problems of varying difficulty.

\begin{table}[H]
  \centering
  \caption{Performance Comparison on 134 Maros-Mészáros Instances with 1000s timelimit and tolerance $\epsilon=10^{-6}$. The best results are marked in \textbf{bold}, and the second best are \underline{underlined}.}
  \medskip
  
  \label{tab:comparison_134}
  \begin{tabular}{lcccc} 
    \toprule
    \textbf{Solver} & \textbf{Total} & \textbf{Solved} & \textbf{SGM10 (s)} & \textbf{Average (s)} \\
    \midrule
    PDQP & \multirow{5}{*}{134} & 118 & 28.53 & 160.70 \\
    PDHCG & & 111 & 33.51 & 210.53 \\
    HPR-QP & & 124 & \textbf{10.56} & \underline{93.32} \\
    PDHCG-II.jl & & \textbf{126} & 17.31 & 105.91 \\
    PDHCG-II-C & & \underline{125} & \underline{12.33} & \textbf{92.79} \\
    \bottomrule
  \end{tabular}
\end{table}

\paragraph{Results} 
Table \ref{tab:comparison_134} shows that PDHCG-II demonstrates superior robustness and efficiency on the Maros-Mészáros benchmark. The Julia and C implementations solve 126 and 125 instances respectively, surpassing HPR-QP (124) and significantly outperforming the original PDHCG (111). PDHCG-II-C also achieves the lowest arithmetic mean runtime of 92.79s. This performance advantage is further amplified on the Mittelmann benchmark (Table \ref{tab:solver_comparison_21}), where PDHCG-II-C solves \textbf{17 out of 21} instances compared to 14 for HPR-QP, while reducing the $\text{SGM}_{10}$ by approximately 43\% (72.74s vs. 126.88s). Notably, PDHCG-II exhibits exceptional stability at high precision; when the tolerance is tightened to $\epsilon=10^{-8}$, it maintains its success rate of 17 solved instances with only a moderate increase in runtime, confirming that the proposed algorithmic enhancements effectively prevent the convergence tail often observed in first-order methods.

\begin{table}[H]
  \centering
  \caption{Performance Comparison on 21 Mittelmann Instances with 3600s timelimit. The best results are marked in \textbf{bold}, and the second best are \underline{underlined}.}
  \label{tab:solver_comparison_21}
  \medskip
  
  \begin{tabular}{lcccc} 
    \toprule
    \textbf{Solver} & \textbf{Total} & \textbf{Solved} & \textbf{SGM10 (s)} & \textbf{Average (s)} \\
    \midrule
    \multicolumn{5}{c}{\textit{Tolerance $\epsilon = 10^{-6}$}} \\ 
    \midrule
    PDQP & \multirow{5}{*}{21} & 11 & 410.34 & 1942.89 \\
    PDHCG & & 10 & 486.81 & 1987.11 \\
    HPR-QP & & 14 & 126.88 & 1282.34 \\
    PDHCG-II.jl & & \underline{16} & \underline{103.61} & \underline{959.35} \\
    PDHCG-II-C & & \textbf{17} & \textbf{72.74} & \textbf{766.94} \\
    \midrule
    \multicolumn{5}{c}{\textit{Tolerance $\epsilon = 10^{-8}$}} \\ 
    \midrule
    HPR-QP & \multirow{3}{*}{21} & 14 & 137.42 & 1318.61 \\
    PDHCG-II.jl & & \underline{16} & \underline{134.29} & \underline{994.05} \\
    PDHCG-II-C & & \textbf{17} & \textbf{91.22} & \textbf{894.32} \\
    \bottomrule
  \end{tabular}
\end{table}

\subsection{Effectiveness of Adaptive Inner Tolerance}

To evaluate the impact of the heuristic adaptive stopping criterion derived in Section 3, we conducted an ablation study on the Maros-Mészáros dataset. Since the inner solver is only invoked when the quadratic matrix $Q$ is non-diagonal, we filtered the dataset to exclude problems with diagonal, resulting in a subset of instances where the efficiency of the inner solver is critical. Both variants were tested with a global termination tolerance of $\epsilon_{\text{tol}} = 10^{-6}$ and a time limit of 1000 seconds per instance. The performance comparison is summarized in Table~\ref{tab:adaptive_inner}. The metrics reported include the number of solved instances, and the Shifted Geometric Mean (SGM) and Arithmetic Mean of both runtime and total iterations.

\begin{table}[H]
    \centering
    \caption{Comparison of Adaptive Inner Tolerance Strategies on Non-Diagonal Maros-Mészáros Problems ($\epsilon=10^{-6}$, Time Limit = 1000s).}
    \label{tab:adaptive_inner}
    \medskip
    
    \begin{tabular}{lccccc}
        \toprule
        & \textbf{Solved} & \multicolumn{2}{c}{\textbf{Time (s)}} & \multicolumn{2}{c}{\textbf{Total Iterations}} \\
        \cmidrule(lr){3-4} \cmidrule(lr){5-6}
        Solver & (Count) & SGM10 & Average & SGM10 & Average \\
        \midrule
        PDHCG (Old version) & 76 & 11.04 & 51.94 & 18,434 & 472,388 \\
        \name & \textbf{78} & \textbf{8.17} & \textbf{27.79} & \textbf{8,231} & \textbf{111,689} \\
        \bottomrule
    \end{tabular}
\end{table}

The proposed adaptive strategy significantly outperforms the original heuristic used in the first-generation PDHCG. It solves 2 additional hard instances and delivers a substantial speedup: SGM runtime is reduced by \textbf{26\%} (11.04s to 8.17s) and mean runtime is nearly halved. Iteration complexity sees an even greater reduction, with SGM iterations dropping by over 50\% and the mean decreasing by a factor of 4. These results confirm that the new criterion, derived from energy dissipation analysis, corrects the conservative nature of the original heuristic. By achieving a tighter balance between inner-loop effort and outer-loop convergence, it prevents unnecessary computations without compromising robustness.

\subsection{Real-world Large-scale QP Problems}
\label{sec:realworldqp}

In this section, to further evaluate the capability of our algorithm in solving large-scale problems, we consider the Lasso problem, with data sourced from real-world datasets, specifically LIBSVM \citep{chang2011libsvm} and the SuiteSparse Matrix Collection \citep{davis2011university}. All selected problems contain more than 6,000,000 nonzero elements, with the largest problem containing nearly 400,000,000 nonzeros. 

The formulation of the Lasso problem is given by:
\begin{equation}
  \min_x \quad \| A x - b \|_2^2 + \lambda \| x \|_1,
\end{equation}
which is equivalent to the following QP problem when considering its dual form:
\begin{equation}
\begin{array}{ll}
\min & y^T y + \lambda \mathbf{1}^T t \\
\text { s.t. } & y = A x - b \\
& -t \leq x \leq t.
\end{array}
\end{equation}
In the experiments, we set the last column of matrix $A$ as vector $b$, and we set the parameter $\lambda = 0.01 \|A^T b\|_{\infty}$ for all test cases. We set $\epsilon = 10^{-6}$ for all instances, which is used in the experiment of PDHCG~\citep{pdhcg}.
\begin{table}[htbp]

  \centering
  \renewcommand{\arraystretch}{1.5} 
  \caption{Solving time (s) over large-scale LASSO problems in LIBSVM and SuiteSparse Matrix Collection. The best results are marked in \textbf{bold}, and the second best are \underline{underlined}. "t" denotes time limit exceeded ($>7200$s).}
  \medskip
  
  \resizebox{\textwidth}{!}{
  \begin{threeparttable}
    \begin{tabular}{c|ccc|ccccccc}
    \toprule
    Problem & $m$ & $n$ & Sparsity & \textbf{PDHCG-II-C} & HPR-QP & PDHCG & PDQP & SCS(GPU) & OSQP & COPT \\
    \hline
    SLS & 1,748,122 & 62,729 & 6.21E-05 & \textbf{1.96} & \underline{2.09} & 3.35 & 7.30 & 345.09 & 80.32 & 88.21 \\
    rcv1\_test & 677,399 & 47,236 & 1.55E-03 & \textbf{2.37} & \underline{6.21} & 7.12 & 19.54 & t & t & t \\
    avazu-site.tr & 23,567,843 & 1,000,000 & 1.50E-05 & \textbf{1337.05} & 4911.41 & \underline{1377.54} & 5124.82 & t & t & t \\
    avazu-app & 40,428,967 & 1,000,000 & 1.50E-05 & \textbf{217.70} & \underline{753.65} & 1429.55 & 5557.97 & t & t & t \\
    avazu-site & 25,832,830 & 1,000,000 & 1.50E-05 & \textbf{1642.69} & \underline{3213.38} & 4224.95 & t & t & t & t \\
    kddb2010\_test & 748,401 & 1,163,024 & 7.74E-06 & \underline{11.59} & 26.87 & \textbf{11.10} & 46.49 & 490.66 & 255.81 & 69.57 \\
    kdda2010\_test & 510,302 & 20,216,830 & 1.87E-05 & \textbf{9.90} & \underline{29.36} & 61.00 & 148.01 & t & t & t \\
    kddb2010\_train & 19,264,097 & 1,163,024 & 7.97E-06 & \underline{703.36} & 1971.24 & \textbf{387.00} & 1715.50 & t & t & t \\
    kdda2010\_train & 8,407,752 & 20,216,830 & 1.80E-06 & \textbf{341.16} & \underline{842.02} & 2705.94 & t & t & t & t \\
    \hline
    \end{tabular}%
  
  \end{threeparttable}
}
\label{real_lasso}
\end{table}
\paragraph{Performance Analysis}
Table \ref{real_lasso} reports the solving time comparison, where \name (denoted as PDHCG-II-C) demonstrates significant advantages in solving these large-scale sparse QPs. First, regarding efficiency on massive datasets, \name significantly outperforms other solvers on the largest instances, such as the \texttt{avazu} series which contains up to 40 million nonzero elements. A notable example is the \texttt{avazu-app} instance, where \name requires only \textbf{217.70s}, achieving a speedup of approximately \textbf{3.5$\times$} over the second-best solver, HPR-QP (753.65s). Second, the results highlight the superior robustness of our approach; while standard solvers such as SCS (GPU), OSQP, and the commercial solver COPT fail to solve the majority of these large-scale instances within the 2-hour time limit (indicated by "t"), \name successfully converges for all problems to the required precision. Finally, compared to the original PDHCG, our implementation shows consistent improvements, particularly on the harder \texttt{avazu} instances where PDHCG is notably slower. This performance gap underscores the effectiveness of the proposed algorithmic enhancements in handling ill-conditioned, large-scale real-world data.

\section{Conclusion}
\label{sec:conclusion}

In this paper, we presented \name, an enhanced first-order solver tailored for large-scale convex quadratic programming. By integrating Halpern-type acceleration ,a PID-controlled primal-dual weight update and a new adaptive stopping criterion for primal subproblems into the PDHCG framework, \name significantly outperforms both its predecessor and state-of-the-art methods. 

Extensive numerical experiments on standard benchmarks and large-scale real-world Lasso problems demonstrate the superiority of \name. It achieves significant speedups, ranging from $2.5\times$ to $5\times$, over the original PDHCG and consistently outperforms state-of-the-art solvers such as PDQP and HPR-QP in both efficiency and robustness. We release the high-performance CUDA-C version as open-source software to support the broader research community.

Future work will focus on extending \name to distributed multi-GPU environments to handle even larger datasets and improve efficiency. Additionally, we plan to explore the application of these acceleration techniques to Quadratic Constrained QP problems and non-convex QP problems.